\documentclass[12pt]{article}
\usepackage{amsmath, amssymb,amsthm}%
\usepackage[dvips]{graphicx}

\newtheorem{theorem}{Theorem}
\newtheorem{lemm}{Lemma}
\newtheorem{predl}{Proposition}
\newtheorem{sled}{Corollary}

\theoremstyle{definition}

\newtheorem*{zam1}{Remark}

\newcommand{\W}{W_{\mathrm{reg}}}

\newcommand{\vol}{\mathrm{vol}}
\newcommand{\rk}{\mathop{\mathrm{rk}}}

\newcommand{\codim}{\mathop{\mathrm{codim}}}
\renewcommand{\Im}{\mathop{\mathrm{Im}}}
\newcommand{\ch}{\mathop{\mathrm{ch}}}

\begin{document}

\begin{flushleft}
UDK 512.817.72
\end{flushleft}

\vspace{10pt}

\begin{center}
\Large {\textsf{{On tilings defined by discrete reflection
groups}}}

\vspace{5pt}

\textsf{\large {P.\,V.\,Bibikov \footnote{The first author was partially supported by the Moebius Contest
Foundation for Young
Scientists}, V.\,S.\,Zhgoon\footnotemark}}
\end{center}

\footnotetext{The work of the second author is partially supported
by RFBR: 09-01-00287, 09-01-12169.}

\vspace{7pt}

\begin{quote}
The recent articles of Waldspurger and Meinrenken contained the
results of tilings formed by the sets of type $(1-w)C^\circ$, $w\in
W$, where $W$ is a linear or affine Weyl group, and $C^\circ$ is an
open  kernel of a fundamental chamber $C$ of the group  $W$. In this
article we generalize these results to cocompact hyperbolic
reflection groups. We also give more clear and simple proofs of the
Waldspurger and Meinrenken theorems.
\end{quote}

\section{Introduction}

Let $X$ be a simply-connected space of a constant curvature and $W$
be a cocompact discrete reflection group acting on this space. In a
spherical case $X\simeq \mathbb{S}^n\subset V$ is a $n$-dimensional
sphere embedded in $(n+1)$-dimensional Euclidian space $V$ with the
inner product $(\cdot, \cdot)$. In  Euclidian case  $X$ is a
$n$-dimensional affine Euclidian space,
  associated Euclidian vector space we denote by $V$.
In a hyperbolic case  $X\simeq \mathbb{H}^n\subset V$ is a connected
component of $n$-dimensional  hyperboloid embedded in
$(n+1)$-dimensional Minkowskii space with the inner product
$[x,y]=-x_0y_0+x_1y_1+\ldots+x_ny_n$ and defined by the equation
 $[x,x]=-1$. In the corresponding cases we call a group
$W$  spherical,  euclidian or  hyperbolic discrete reflection group.

The distance between points $x,y\in X$ is denoted by $\rho(x,y)$. By
$D$ denote a compact fundamental domain for the action of $W$ on $X$
and by $D^\circ$ denote its interior. In spherical an hyperbolic
cases let $C$ be a fundamental chamber for the group $W$ in the
space $V$ and $C^\circ$ is its interior. (We can assume that
$D=C\cap X$ and $D^\circ=C^\circ\cap X$.) Denote by $X^W$ the space
of $W$-fixed points (we note that $X^W$ is not empty only in the
spherical case) and by $W_x$ we denote a stabilizer of a point $x$ in the
group $W$. By $\W$ we denote a set of elements in  $W$, those fixed
point sets are equal to $X^W$.
 We denote by $s_H$ a reflection in a hyperplane $H$.

This paper is organized as follows. In
section~\ref{razd.lemm.nepodv.t.} we prove a fundamental ''Fixed
point lemma'' for cocompact discrete reflection groups. In
section~\ref{razd.razbieniya}  by means of this lemma we prove the
theorems of Waldspurger~\cite{Wald} and Meinrenken~\cite{Mein} of
tilings formed by the sets of type $(1-w)C^\circ$, where $w\in W$,
and we also prove a corresponding theorem for hyperbolic reflection groups. In the
last section~\ref{razd.svoy.razb.} we consider the properties of
these decompositions, we also relate them with the theorem of
Kostant~\cite{Cart}.

\textbf{Acknowledgements.} The authors are grateful to
E.\,B.\,Vinberg  for highly useful discussions and constant
attention to their work. We want to thank  V.\,L.\,Popov, who drew
our attention to paper~\cite{Mein}.

\section{Fixed point lemma}\label{razd.lemm.nepodv.t.}

\begin{lemm}\label{lemm.Dirihle}
For the points $x_0\in D^\circ$ and $x\in D$ we have inequality
$\rho(x_0,wx)\geqslant\rho(x_0,x)$, besides for every  $w\not\in
W_x$ the inequality is strict.
\end{lemm}

\begin{figure}[h]
\begin{center}
\includegraphics{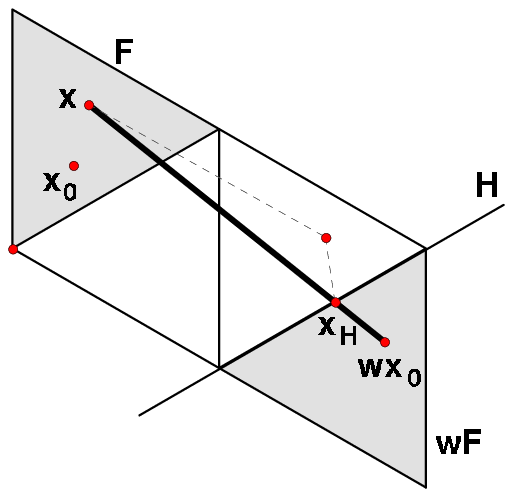}
\caption{}\label{ris.edin}%
\end{center}
\end{figure}

\begin{proof}
It is sufficient to prove the inequality
$\rho(wx_0,x)\geqslant\rho(x_0,x)$. The proof is by the induction on
the length $\ell(w)$ of the element $w$. We can find a wall $H$ of
codimension one in the chamber $wD$ such that $\ell(s_Hw)<\ell(w)$.
This means (cf.~\cite[теорема 1, гл. 5, пар. 3]{Burb}) that $D$ and
$wD$ lie in the different half spaces cut by the hyperplane $H$. Let us
denote by $x_H$  the intersection point of the hyperplane $H$ and a
closed interval connecting $x$ and $wx_0$ (Fig.~\ref{ris.edin}).
Then we have $\rho(x_H,wx_0)=\rho(x_H,s_H wx_0)$. By the induction
assumption and a triangle inequality we have
$$\rho(x,x_0)\leqslant\rho(x,s_Hwx_0) \leqslant \rho(x,x_H)+
\rho(x_H,s_Hwx_0)=\rho(x,wx_0).$$

A triangle inequality becomes an equality iff $x_H$ belongs to the
closed interval $[x,s_Hwx_0]$. This is possible only when $x=x_H$,
that implies $s_Hx=x$. Thus we have the equality
$\rho(x,x_0)=\rho(x,wx_0)$ only when we have the equalities
$\rho(x,s_Hwx_0)=\rho(x,x_0)$ and $s_Hx=x$. From the first equality
and induction step we get $s_Hwx=x$, taking into account  the second
equality we obtain $w\in W_x$.
\end{proof}

\begin{lemm}[Fixed point lemma]\label{lemm.nepodv.tochka}
Let $g\in \mathrm{Isom}(X)$. Then there exists a unique element
$w\in W$ such that the element  $w^{-1}g$ has a fixed point in
$D^\circ$.
\end{lemm}

\begin{proof}
1. The proof is by the induction on $\dim X$. First we prove the
existence of $w\in W$ such that $w^{-1}g$ has a fixed point in the
closed chamber $D$. Since the group  $W$ is cocompact from a
topological viewpoint the chamber $D$ is  a closed ball. Let $x\in
D$ be an arbitrary point.  Then the point  $gx$ belongs  to one of
the chambers $wD$, where $w\in W$. Let us set $f(x)=w^{-1}gx\in D$.
This defines a map $f$ from the ball $D$ into itself. It is clear
that it is well defined (since the chamber  $D$  is compact, the
chamber $gD$ intersects only finite set of chambers of type $wD$)
and continuous (indeed, if  $x_1,x_2\in D$ and $w_1x_1=w_2x_2$ then
$x_1=x_2$). Applying the Brauer fixed point  theorem we get that
this map has a fixed point. Thus we obtain an element $w\in W$ and a
point $x\in D$ such that $w^{-1}gx=x$.

For $x\in D^\circ$, there is nothing to prove. Let $x\in D\setminus
D^\circ$. Consider a sphere $S(x)$ with the center in the point $x$,
which doesn't intersect the hyperplanes spanned by the faces of $D$
that do not contain $x$. We also consider the subgroup $W_x\subset
W$ that acts on the sphere $S(x)$. The group $W_x$ is generated by
the reflections that fix the point $x$. The chamber $S(x)\cap D$ is
fundamental domain for the group $W_x$, since it is cut by those
hyperfaces of the chamber $D$ that contain $x$. By the induction
hypothesis applied  to the sphere $S(x)$, the element $w^{-1}g$, the
group $W_x$ and the chamber $S(x)\cap D$ there exists an element
$w'\in W_x$ and a point $x'\in (S(x)\cap D)^\circ$ such that
$w'^{-1}w^{-1}gx'=x'$. Since the sphere $S(x)$ does not intersect
the hyperplanes spanned by those faces of the chamber $D$ that do
not contain $x$ we have
 $x'\in D^\circ$.

2. Let us prove the uniqueness of $w$. Suppose there exist two
elements $w_1$, $w_2\in W$ such that $w_1^{-1}gx_1=x_1$ and
$w_2^{-1}gx_2=x_2$, where $x_1$, $x_2\in D^\circ$. Let us set
$w=w_1^{-1}w_2$, then we have
$\rho(gx_1,gx_2)=\rho(w_1x_1,w_1wx_2)$. If $w\neq 1$ taking into
account that $W_{x_i}=1$ we get the following inequalities from
Lemma~\ref{lemm.Dirihle}:
$$\rho(gx_1,gx_2)=\rho(x_1,x_2) < \rho(x_1,wx_2)=\rho(w_1x_1,w_1wx_2).$$
Thus we have $w=1$ and $w_1=w_2$.
\end{proof}

\begin{zam1}
In a more general case when we consider the fundamental domains of
finite volume (i.e. the chamber $D$  can have infinite points)
Lemma~\ref{lemm.nepodv.tochka} is not true. Indeed,
 consider a hyperbolic space $X$
and a parallel  transport $g$ along the line containing the infinite
point $o$  of the chamber $D$. Then there are no elements $w\in W$
such that $gx=wx$ for $x\in D^\circ$. Assume the contrary, then the
point $gx=wx$ is contained in the horosphere $\mathcal{O}$ with the
center in $o$ that also contains $x$, that is not true since
$g\mathcal{O}\cap\mathcal{O}=\varnothing$. Moreover one can prove
that the equality $gx=wx$ doesn't hold for all points $x\in D$.
\end{zam1}

\begin{sled}\label{sled.|g|<=1}
Let $g\colon X\to X$  be a continuous map.  Then there exists an element $w\in W$
such that the transformation $w^{-1}g$ has a fixed point in $D$. Moreover, if we have
$\rho(gx_1,gx_2)\leqslant\rho(x_1,x_2)$ for all $x_1,x_2\in X$ and the fixed point belongs
 to $D^\circ$ then such element $w$ is unique.
\end{sled}

\begin{proof}
The first claim follows from the proof of part 1 of ``Fixed point
lemma''~\ref{lemm.nepodv.tochka}. The second follows from
Lemma~\ref{lemm.Dirihle} and the proof of second part of  ``Fixed
point lemma''.
\end{proof}


\section{Tilings related to discrete reflection groups}\label{razd.razbieniya}

In the spherical and hyperbolic cases let us define $\Pi$  as a set
of unite normals to the hyperfaces of the chamber $C$ lying in the
same halfspaces as $C$. For every $e\in\Pi$  denote by $H_e$ a
hyperface orthogonal to $e$ and by $s_e$ a reflection in $H_e$. By
$C^*$  we denote the cone dual to $C$ with respect to the
corresponding inner product.

In a hyperbolic case we consider the cone $\mathcal{K}=\{x\in
V
:[x,x]\leqslant 0\}=\mathcal{K}_+\cup \mathcal{K}_-$,  where
$\mathcal{K}_+=\{x=(x_0,x_1,\ldots,x_n)\in\mathcal{K}:x_0\geqslant0\}$.
Let us recall that in this case $[e,e]>0$ for all $e\in\Pi$, we
also have the inclusions $C\subseteq \mathcal{K}_+$ and
$\mathcal{K}_-\subset C^*$ (cf.~\cite[ex. 12--13 to $\S$4 ch.
5]{Burb}).

\begin{lemm}\label{sled.(1-w)C<C^*}
For  spherical or  hyperbolic reflection group $W$ we have the inclusion $$C^*\supseteq \bigsqcup\limits_{w\in
W}(1-~w)C^\circ.$$
\end{lemm}

\begin{proof}
By  Lemma~\ref{lemm.Dirihle} we have $\rho(wx_0,x)\geqslant\rho(x_0,x)$ where
$x_0\in C^\circ$ and $x\in C$. In the case of spherical reflection group
$W$ we have:
$$\frac{(wx_0,x)}{\sqrt{(wx_0,wx_0)(x,x)}}=\cos\rho(wx_0,x)\leqslant
\cos\rho(x_0,x)=\frac{(x_0,x)}{\sqrt{(x_0,x_0)(x,x)}},$$  that implies
that $(x,(1-w)x_0)\geqslant0$. In the case of  hyperbolic group $W$ the argument is similar (we have to remind the following
formula for the distance in $\mathbb{H}^n$:
$\ch\rho(x_0,x)=-\frac{[x_0,x]}{\sqrt{[x_0,x_0][x,x]}}$).
\end{proof}

\begin{theorem}[Waldspurger {\cite{Wald}}]\label{th.Wald.}
Let $W$ be a spherical reflection group. We have a decomposition
$C^*=\bigsqcup\limits_{w\in W}(1-w)C^\circ.$
\end{theorem}

\begin{proof}
By ``Fixed point lemma''~\ref{lemm.nepodv.tochka} applied to the element $g=s_u$,
 where $u\in C^*$,  there exists a point $x\in
C^\circ$  and a unique element $w\in W$ such that $w^{-1}s_ux=x$.
Thus we have $(1-w)x=\frac{2(u,x)}{(u,u)}u$. Setting
$v=\frac{(u,u)}{2(u,x)}x$, we get the equality $(1-w)v=u$.
 Since $x\in C^\circ$ and $(u,x)>0$ we get $v\in C^\circ$. Taking into account
Lemma~\ref{sled.(1-w)C<C^*} this means that
$\bigsqcup\limits_{w\in W}(1-w)C^\circ=C^*$.
\end{proof}

\begin{theorem}
Let $W$ be a Euclidian reflection group. Then for every $h\in
\mathrm{Isom}(X)$ we have an equality $V=\bigsqcup\limits_{w\in
W}(h-w)D^\circ$.
\end{theorem}

\begin{proof} By Corollary~\ref{sled.|g|<=1} applied to the element
$g=t_{-u}h$ (where $t_{-u}$ is a parallel transport by the vector
$-u$) there exists a point $v\in D$ and a unique element $w\in W$
such that $w^{-1}t_{-u}hv=v$. This is equivalent  to $(h-w)v=u$.
\end{proof}

\begin{sled}[Meinrenken, {\cite[Thm.2]{Mein}}]\label{th.Mein.}
There is an equality $V=\bigsqcup\limits_{w\in
W}(1-w)D^\circ$.
\end{sled}

\begin{sled}\label{sled.(h-w)A}
For every $h\in GA(X)$  there is an equality $V=\bigcup\limits_{w\in
W}(h-w)D$, moreover if $\rho(hx_1,hx_2)\leqslant\rho(x_1,x_2)$ for
all $x_1,x_2\in X$ the chamber $(h-w)D$ have  pairwise intersections
of codimension at least 1.
\end{sled}

\begin{proof}
By Corollary~\ref{sled.|g|<=1} applied to the element $g=t_{-u}h$
there exists a point $v\in D$ and an element $w\in W$ such that
 $w^{-1}t_{-u}hv=v$. Thus we have $(h-w)v=u$. Moreover the element $w$
is unique if  $v\in D^\circ$, that implies the second assertion.
\end{proof}

\begin{zam1}
Consider the spherical reflection subgroup $W_0\subset W$. From
Theorem~\ref{th.Mein.} we get that the set $M=\bigsqcup\limits_{w\in
W_0}(1-w)D^\circ$ is a fundamental domain for the subgroup of $W$,
generated by the parallel transports correspond to the simple
coroots. The set $N=\bigcup\limits_{w\in W_0}wD$ is the closure of
another fundamental domain for the subgroup of parallel transports
in consideration. Thus we have $\sum\limits_{w\in
W_0}\det(1-w)\cdot\vol(D)=\vol(M)=\vol(N)=|W_0|\cdot\vol(D)$, that
implies the formula
 $\sum\limits_{w\in W_0}\det(1-w)=|W_0|$
(cf.~\cite[ex.3,\S2,ch.5]{Burb}). By the same argument one
obtains $\sum\limits_{w\in W_0}\det(1-hw^{-1})=|W_0|$, where
$\rho(hx_1,hx_2)\leqslant\rho(x_1,x_2)$ for all $x_1,x_2\in X$.
\end{zam1}

\begin{theorem}\label{th.Neevkl.}
Let $W$  be a hyperbolic reflection group. We have the following equalities:
(i) $C^*\setminus\mathcal{K}_-=\bigsqcup\limits_{w\in W}(1-w)C^\circ$
(ii) $\mathcal{K}_-^\circ=\bigsqcup\limits_{w\in W}(-1-w)C^\circ$.
\end{theorem}

\begin{proof}
$(i)$ First let us prove the inclusion <<$\supseteq$>>. By
Lemma~\ref{sled.(1-w)C<C^*} we have
$C^*\supseteq\bigsqcup\limits_{w\in W}(1-w)C^\circ$. Let us prove
that $\bigsqcup\limits_{w\in
W}(1-w)C^\circ\not\subset\mathcal{K}_-$. The proof is by induction
on the length  $\ell$ of a reduced decomposition of $w$.
 The case $\ell(w)=0$ is obvious. If
$\ell(w)>0$ we have an equality $w=s_ew'$, where $\ell(w')=\ell(w)-1$.
The latter means that the cones
 $C$ and $w'C$ lie in the same halfspace of $H_e$, and in particular $[w'x,e]>0$.
We obtain
$$[(1-w)x,(1-w)x]=[(1-w')x, (1-w')x]+4\frac{[w'x,e][x,e]}{[e,e]}>0$$
for all $x\in C^\circ$. This proves that
$C^*\setminus\mathcal{K}_-\supseteq\bigsqcup\limits_{w\in
W}(1-w)C^\circ$.

Let us prove the opposite inclusion. Consider $u\in C^*\setminus \mathcal{K}_-$.
Let us apply ``Fixed point lemma''~\ref{lemm.nepodv.tochka} to the element $g=s_u$.
 We obtain: $wx=s_ux=x-2\frac{[u,x]}{[u,u]}u$,
hence $(1-w)x=2\frac{[u,x]}{[u,u]}u$. Setting
$v=\frac{[u,u]}{2[u,x]}x$ we get $(1-w)v=u$. It is obvious that
$[u,x]>0$, since $u\not\in \mathcal{K}_-$ then $[u,u]>0$. Thus
$v\in C^\circ$ that proves the desired decomposition.

$(ii)$ First let us prove the inclusion <<$\supseteq$>>. As
previously we argue by the induction on the length $\ell$  of the
shortest decomposition of $w$. When $\ell(w)=0$ we are done. If
$\ell(w)>0$ we have $w=s_ew'$, where $\ell(w')=\ell(w)-1$ and
$$[(-1-w)x,(-1-w)x]=[(-1-w')x,
(-1-w')x]-4\frac{[w'x,e][x,e]}{[e,e]}<0$$ for all $x\in C^\circ$.
Besides it is obvious that $(-1-w)x\not\in \mathcal{K}_+$ for $x\in
C^\circ$. This implies
$\mathcal{K}_-^\circ\supseteq\bigsqcup\limits_{w\in
W}(-1-w)C^\circ$.

Let us prove the opposite inclusion. Consider $u\in \mathcal{K}_-^\circ$.
This time  applying ``Fixed point lemma''~\ref{lemm.nepodv.tochka} to the element $g=-s_u$,
we get: $wx=-s_ux=-x+\frac{2[u,x]}{[u,u]}u$,
thus  $(-1-w)x=-\frac{2[u,x]}{[u,u]}u$. Setting
$v=-\frac{[u,u]}{2[u,x]}x$, we obtain $(-1-w)v=u$. We have
$[u,x]>0$, and since $u\in \mathcal{K}_-^\circ$, we have $[u,u]<0$. This implies
$v\in C^\circ$ and proves the second decomposition.
\end{proof}

\begin{zam1}
Theorem~\ref{th.Neevkl.} cannot be generalized directly to the
discrete hyperbolic groups with a fundamental domain of finite
volume. We shall construct a  point  $u\in
C^*\setminus\mathcal{K}_-$ such that the equality  $u=(1-w)v$  is
impossible for all $v\in C^\circ$ (and even for $v\in C$). Consider
$u\in (C^*\setminus\mathcal{K}_-)\cap \langle s_eo,o\rangle$, where
$e\in w^{''}\Pi$,  $w^{''}\in W$ and $o\in C$  is an infinite point,
moreover $H_u\cap C=\varnothing$. Let us prove that the equality
$u=(1-w)v$ is impossible for every $v\in C^\circ$.

Indeed there is $\lambda>0$ such that $\lambda s_es_uo=o$. Since
$H_u\cap C=\varnothing$, then $\lambda<1$. This implies that if
$\mathcal{O}$ is a horosphere  with the center in $o$, then
$s_es_u\mathcal{O}$ lies inside the horosphere $\mathcal{O}$. Thus
we get that $u=(\lambda-w')p$, where $w'=s_e$ and
$p=\frac{[u,u]}{2[u,o]}o\in C$. Assume that there exist $w\in W$ and
$v\in C^\circ$ such that $u=(1-w)v$. Then we get
$(\lambda-w')p=(1-w)v$, which implies $\lambda p-v=w'p-wv$. Taking
the norm of this equality we obtain $(v,(\lambda-\tilde{w})p)=0$,
where $\tilde{w}=w^{-1}w'$. But we have
$(\lambda-\tilde{w})p=(1-\tilde{w})p-(1-\lambda)p\in C^*$ and the
equality $2(v,(\lambda-\tilde{w})p)=0$ is possible  only if
$(\lambda-\tilde{w})p=0$. But this is impossible due to limit
argument. Indeed consider the sequence $\{p_n\}\subset C^\circ$ such
that $p_n\to p$. If $x\in C^\circ$ then
$0>[p_n,x]\geqslant[\tilde{w}p_n,x]$ and
$\lambda|[p_n,x]|<|[p_n,x]|\leqslant|[\tilde{w}p_n,x]|$. Thus
$|[\lambda p_n,x]|<|[p_n,x]|\leqslant|[\tilde{w}p_n,x]|$. Taking the
limit as $n\to\infty$, we obtain $|[\lambda
p,x]|<|[p,x]|\leqslant|[\tilde{w}p,x]|$, that contradicts the
assumptions.

Thus we found a subset in
$C^*\setminus\mathcal{K}_-$  that is not covered by the cones of type
$(1-w)C^\circ$. The similar arguments also hold for the decomposition of $\mathcal{K}_-$.
\end{zam1}

\section{Properties of decompositions}\label{razd.svoy.razb.}

In this paragraph we restrict ourselves to the case of finite
reflection groups acting on Euclidian space $V$. In the preceding
paragraphs we obtained various theorems about decompositions formed
by the sets of type $(1-w)C^\circ$, where $w\in W$.  One of the main
questions related with these decompositions is the question which
cones of type $C_w=(1-w)C$ are \emph{adjacent}.

For the case when  $w$, $w'\in\W$, the answer was given
in~\cite{BiZh} (Figures ~\ref{ris.A_3} and \ref{ris.B_3} describe
the sections of corresponding cones in the
 cases $W=A_3$ and $W=B_3$).  In case of finite reflection groups this question is related to the
 original proof of Walspurger theorem~\ref{th.Wald.} in~\cite{Wald}. The main idea of the proof is to construct from
 an element $w\in
\W$ and a vector $e\in\Pi$ the element $w'\in\W$ such that the cones $C_w$
and $C_{w'}$ are adjacent and have intersection of codimension one contained in the  wall
$(1-w)H_e$.

\begin{figure}[h]
\begin{center}
\parbox{8cm}{%
\includegraphics{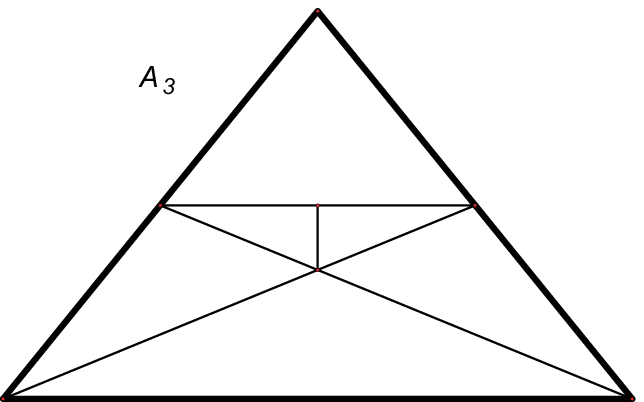}
\caption{}\label{ris.A_3}%
}\qquad
\parbox{8cm}{%
\includegraphics{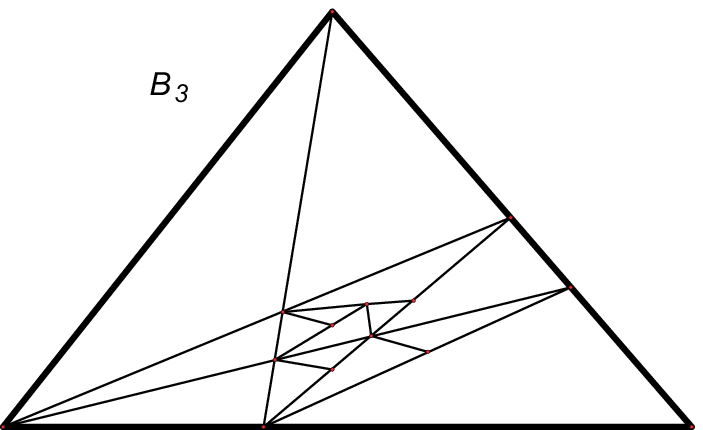}
\caption{}\label{ris.B_3}%
}
\end{center}
\end{figure}

Before giving the answer to this question let us state two simple
facts.

\begin{predl}\label{predl.ws_e><w}
Let $g\in O(V)$ be an arbitrary orthogonal transformation and let $s\in O(V)$ be a reflection.
 Then the dimensions of fixed point sets of  $g$ and
$gs$ differ exactly by 1.
\end{predl}

\begin{proof}
The eigenvalues' norm of orthogonal transformation $g$ is equal to
1, and each complex eigenvalue has a corresponding conjugated
eigenvalue (counted with multiplicities). Thus we have $\det
g=(-1)^{n-k}$, where $k=\dim \ker(1-g)$. Since $\det s=-1$ we get
$\det g=-\det(gs)$. This implies that $\dim \ker(1-g)\neq \dim
\ker(1-gs)$.

Without the loss of generality assume that $\dim \ker(1-g)> \dim \ker(1-gs)$. On the other hand
 $\ker(1-g)\cap H \subset \ker(1-gs)$ (where $H$
is the reflection hyperplane) that implies $\ker(1-g)-1\leqslant \ker(1-gs)<
\ker (1-g)$ and $\ker(1-gs)=\ker(1-g)-1$.
\end{proof}

\begin{predl}
If $wC\cap w'C\neq\{0\}$, then $C_w\cap C_{w'}\neq\{0\}$.
\end{predl}

\begin{proof}
Indeed the equality  $wx=w'x'$ for $x,x'\in C\setminus\{0\}$ implies that
$x=x'$ and $(1-w)x=(1-w')x'$.
\end{proof}

For each  $e\in\Pi$ denote by $\pi_e$ the vector orthogonal to
$V^W$ such that its inner product with  $f$ is equal to Kroneker symbol
$\delta_{ef}$ for all $f\in\Pi$. Let us fix an element $w\in\W$, for each
vector $u\in V$  we denote by $v_u$ the vector orthogonal to
$V^W$ such that $(w^{-1}-1)v_u=u$.

\begin{lemm}[cf.~{\cite[Lemma 3]{BiZh}}]\label{svoystva}
Let $u,r\in V$ and $e\in \Pi$. For vectors $v_{u},v_{r}$ and $v_{e}$  we have the following properties:
1) $(v_{u},r)+(v_{r},u)=-(u,r)$; 2) $(v_{u},u)=-\frac 1 2(u,u)$; 3)
$v_{u}=ws_{u}v_{u}$; 4) $v_u\bot (1-w)H_u$;  5)
$\ker(1-ws_u)=\langle v_u\rangle$, 6) $(v_e,(1-w)\pi_e)<0$;.
\end{lemm}

\begin{proof}
1) Taking inner product of the equalities   $w^{-1}v_u=v_u+u$ and $w^{-1}v_r=v_r+r$ we obtain
$(v_u,v_r)=(w^{-1}v_u,w^{-1}v_r)=(v_u,v_r)+(v_u,r)+(v_r,u)+(u,r)$, that proves the equality.
 3) From 2)
we get that $w^{-1}v_u=v_u+u=v_u-\frac{2(v_u,u)}{(u,u)}u=s_uv_u$.
The proofs of other claims are evident.
\end{proof}

\begin{theorem}[cf.~{\cite[Theorem 2]{BiZh}}]
Let $w\in\W$. The element  $w'\in W$  is regular and the cones  $C_w$ and
$C_{w'}$ are adjacent  iff there exist  vectors $e,f\in\Pi$ such that $w'=ws_es_f$ and $(v_e,f)>0$; in this case
 $\codim (1-w)H_e\cap(1-w')H_f=1$.
\end{theorem}

Figures~\ref{ris.sm.kon.1} and \ref{ris.sm.kon.2}  show the
transversal sections of the cones considered in the theorem.

\begin{figure}[h]
\begin{center}
\parbox{7cm}{%
\includegraphics{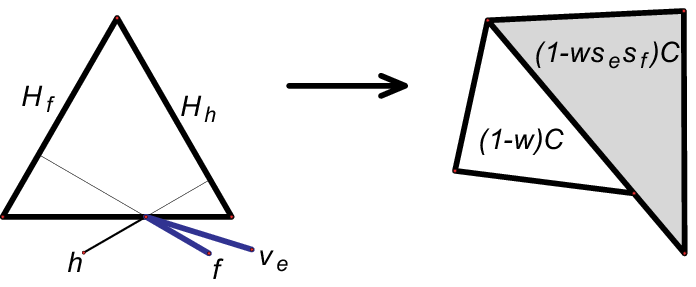}
\caption{}\label{ris.sm.kon.1}%
}\quad\quad
\parbox{9cm}{%
\includegraphics{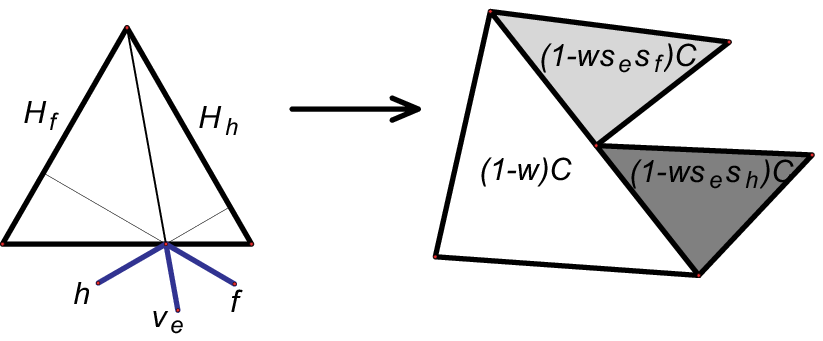}
\caption{}\label{ris.sm.kon.2}%
}
\end{center}
\end{figure}

Let us consider the case when the cone of full dimension is adjacent
to the cone of strictly smaller dimension. Let $\dim C_{w'}<\dim
C_w=n$. We say that the cone $C_{w'}$ is \emph{adjacent}  to the
cone $C_w$ if $\dim C_{w'}=\dim(C_{w'}\cap C_{w})$.

\begin{lemm}\label{lemm.dim.ker}
If $w=w'\tilde{w}\in\W$ and the cone $C_{w'}$ is adjacent to $C_w$ then we have
$\rk(1-w')+\rk(1-\tilde{w})=n$.
\end{lemm}

\begin{proof}
Consider the minimal set $\{\pi_{i_1},\ldots,\pi_{i_k}\}$
such that $C_w\cap C_{w'}\subseteq(1-w)\langle
\pi_{i_1},\ldots,\pi_{i_k}\rangle_+=(1-w)\widetilde{C}$. This implies
 $\dim\ker(1-w')\geqslant n-\dim(1-w)\widetilde{C}=n-k$.

There exist $v\in \widetilde{C}^\circ$ and $v'\in C^\circ$ such that
$(1-w)v=(1-w')v'\in C_{w'}$. That implies $v-v'=w'(\tilde{w}v-v')$.
Taking the scalar square of this equality after simplifications we obtain:
$((1-\tilde{w})v,v')=0$. Since $v'\in C^\circ$ and $(1-\tilde{w})v\in
C^*$ then we have $v\in\ker(1-\tilde{w})$. By Steinberg fixed point theorem
 $\langle
\pi_{i_1},\ldots,\pi_{i_k}\rangle\subseteq \ker(1-\tilde{w})$ that
 implies $\dim\ker (1-\tilde{w})\geqslant k$ and
$\dim\ker(1-w)+\dim\ker(1-\tilde{w})\geqslant (n-k)+k=n$.

For the proof of the opposite inequality it is sufficient to note
that $\ker(1-w')\cap\ker(1-\tilde{w})=\{0\}$. Indeed if $x\in
\ker(1-w')\cap\ker(1-\tilde{w})$ then we have $wx=w'(\tilde{w}x)=x$
and $x\in\ker(1-w)=\{0\}$. That implies
$\dim\ker(1-w')+\dim\ker(1-\tilde{w})\leqslant n$.
\end{proof}

Let $R$  be the set of the unite normals to the reflection
hyperplanes of group $W$. It will be convenient for us to consider
the decompositions of the elements of $W$ in the products of the
reflections $s_u$ for $u\in R$ where the reflection is not supposed
to be simple (i.e. equal to $s_e$ for $e\in\Pi$). The next theorem
of Kostant (cf.~\cite[prop.5.1]{Cart}) describes this type of
decompositions.

\begin{theorem}[Kostant]\label{th.Kostant}
Let $w=s_{u_1}s_{u_2}\ldots s_{u_k}$ where ${u_1},{u_2}\ldots u_k
\in R$. Then the following assertions are equivalent: (i) $\rk (1-w)=k$;
(ii) $w=s_{u_1}s_{u_2}\ldots s_{u_k}$ is the decomposition of minimal length;
 (iii) $u_1,{u_2}\ldots u_k$ are linearly independent.
\end{theorem}

\begin{proof}
(i)$\Rightarrow$(ii) and (i)$\Rightarrow$(iii). Since $\ker
(1-w)\supseteq H_{u_1}\cap H_{u_2}\cap\ldots \cap H_{u_k}$ we have
$k\geqslant\rk (1-w)$. In case of linear dependence of
$u_1,{u_2}\ldots u_k$ the inequality is strict.

(ii)$\Rightarrow$(i). Arguing by the induction on $\dim V$ let us show that there
exists a decomposition of length not bigger than $\rk (1-w)$.
 We may assume that the theorem is proved for all $w\notin \W$. Indeed let $V^w\neq 0$
 be the subspace of $w$-fixed vectors, we may apply the induction step to the subspace
  $(V^w)^{\bot}$ and the subgroup $W_{V^w}\subset W$ that fixes $V^w$ pointwise.

 Let $w\in \W$. Let us choose an arbitrary vector $u\in R$. Since  $w\in \W$ from Proposition~\ref{predl.ws_e><w}
we obtain $\dim \ker(1-ws_u)>0$ (this also follows from Lemma~\ref{svoystva}, 3)).
 Thus we obtain $ws_u\notin \W$  this assumes that the claim
  is proved for such elements, we are finished.

(iii)$\Rightarrow$(i). Assume that  $u_1,{u_2}\ldots u_k$ are linearly
independent. Let us choose $x_1 \in (H_{u_2}\cap\ldots \cap
H_{u_k})\setminus H_{u_1}$. Since
$(1-w)x_1=\frac{2(x_1,u_1)}{(u_1,u_1)}u_1$ we have $u_1\in (1-w)V$.
Taking $x_2\in (H_{u_3}\cap\ldots \cap H_{u_k})\setminus
H_{u_2}$ we obtain
$(1-w)x_2=c_{21}u_1+\frac{2(x_2,u_2)}{(u_2,u_2)}u_2$ that implies
$u_1,u_2\in (1-w)V$. Taking $x_i\in (H_{u_{i+1}}\cap\ldots \cap
H_{u_k})\setminus H_{u_{i}}$ we get
$(1-w)x_i=(1-s_{u_1}\ldots
s_{u_i})x_i=c_{i1}u_1+c_{i2}u_2+\ldots+c_{i-1i}u_{i-1}+\frac{2(x_i,u_i)}{(u_i,u_i)}u_i$,
откуда $u_i\in(1-w)V$. The latter implies $\dim (1-w)V\geqslant k$ and finishes the proof.
\end{proof}

The decomposition of $w$ into the product of reflections that
satisfy the conditions of the Kostant theorem is called
\emph{minimal}.

\begin{theorem}
Suppose we are given $w=w'\tilde{w}\in\W$ such that $\rk(1-w')=k$ and the decomposition $w'=s_{u_1}\ldots
s_{u_k}$ is minimal. The cone  $C_{w'}$ is adjacent to $C_{w}$
iff $\tilde{w}=s_{u_{k+1}}\ldots
s_{u_n}$ where $u_1$, \ldots, $u_n$  are linearly independent and
$u_{k+1},\ldots u_n\in  \widetilde{R}=R\cap \Im(1-\tilde{w})$. In particular
if the cone $C_w$ is adjacent to $C_{w'}$ then
$\tilde{w}\in\widetilde{W}_{\mathrm{reg}}$ where $\widetilde{W}$
is a reflection group generated by the reflections $s_u$ for $u\in
\widetilde{R}$.
\end{theorem}

\begin{proof}
By Lemma~\ref{lemm.dim.ker} we have:
$\rk(1-\tilde{w})=n-\rk(1-w')=n-k$. Applying  the Kostant theorem we
can find the elements
 $u_{k+1}$, \ldots, $u_{n}$ such that
$\tilde{w}=s_{u_{k+1}}\ldots s_{u_n}$ and the elements $u_1$, \ldots, $u_n$
are linearly independent.

Let $v\in C$ and $v'\in C^\circ$ be the vectors that satisfy
$(1-w)v=(1-w')v'$. From Theorem~\ref{th.Kostant} (iii) it follows
that $v\in\ker(1-\tilde{w})=H_{u_{k+1}}\cap\ldots\cap H_{u_n}$. From
the condition
 $\dim C_{w'}=\dim(C_w\cap C_{w'})$ and the proof of Lemma~\ref{lemm.dim.ker}
 it follows that $\widetilde{C}\subset H_{u_i}$
for all $i=k+1,\ldots, n$. That implies $u_{k+1},\ldots u_n\in
\widetilde{R}$ and proves the claim.

In opposite direction, if $\tilde{w}=s_{u_{k+1}}\ldots s_{u_n}$,
where $u_1$, \ldots, $u_n$ are linearly independent,  then by
theorem of Kostant $\rk(1-w)=n$. Since $u_{k+1},\ldots u_n\in
\widetilde{R}$, we have $\widetilde{C}\subset H_{u_{k+1}}\cap\ldots
\cap H_{u_n}$ and $\dim (H_{u_{k+1}}\cap\ldots \cap H_{u_n}\cap
C)=k$. For every $x\in H_{u_{k+1}}\cap\ldots H_{u_n}\cap C$ we get
$(1-w)x=x-w'(s_{u_{k+1}}\ldots s_{u_n}x)=(1-w')x$ that implies $\dim
C_{w'}=\dim(C_{w'}\cap C_{w})$.
\end{proof}


e-mails: tsdtp4u@proc.ru, zhgoon@mail.ru


\begin{thebibliography}{99}
\bibitem{Wald}Waldspurger\,J.-L. ''\textsl{Une remarque sur les syst\`emes de
racines}'' // Journal of Lie Theory, Volume 17 (2007), Number 3. P.
597--603.

\bibitem{Mein}Meinrenken\,E. ''\textsl{Tilings Defined by Affine Weyl
Groups}''; arXiv : math/0811.3880v2, 2009.

\bibitem{Cart} Wang\,H-C., Pasiencier\,S. ''\textsl{Commutators in a Semisimple Lie
Group}'' // Proc, Amer. Math. Soc., Volume 13 (1962), 907--913.

\bibitem{Burb}Bourbaki\,N. \textsl{Lie Groups and Lie Algebras: Chapters 4-6 (Elements of Mathematics)} Springer-Verlag, 2002.

\bibitem{BiZh}Bibikov\,P.\,V., Zhgoon\,V.\,S. ''\textsl{On the Waldspurger theorem}'' // Russian Mathematical Surveys, (2009).

\end{thebibliography}
\end{document}